\documentclass[11pt,letterpaper,reqno]{amsart}

\usepackage{amsmath}
\usepackage{amsfonts}
\usepackage{amssymb}
\usepackage{amsthm}
\usepackage{amscd}
\usepackage[hidelinks]{hyperref}
\usepackage{color,soul}
\usepackage{centernot}

\usepackage{todonotes}

\usepackage{bbm}
\usepackage{latexsym}
\usepackage{mathrsfs}
\usepackage{psfrag}
\usepackage[dvips]{epsfig}
\usepackage{epsfig}
\usepackage[all]{xy}         

\theoremstyle{plain}                              

\newtheorem{theorem}{Theorem}[section]
\newtheorem{cor}[theorem]{Corollary}
\newtheorem*{conj}{Conjecture}
\newtheorem{proposition}[theorem]{Proposition}
\newtheorem{lemma}[theorem]{Lemma}

\theoremstyle{definition}                         
\newtheorem*{defn}{Definition}
\newtheorem{example}{Example}[section]
\newtheorem*{remark}{Remark}

\theoremstyle{remark}                             

\numberwithin{equation}{section}

\newcommand{\R}{\mathbb{R}}                     
\newcommand{\C}{\mathbb{C}}                     
\newcommand{\N}{\mathbb{N}}                     

\newcommand{\Ca}{\mathop{\mathrm{Cap}}} 
\newcommand{\leb}{\mathop{\mathrm{Leb}}} 

\newcommand{\muo}{\mu_{[0,1]}}
\newcommand{\hf}{\widehat{f}}
\newcommand{\wmu}{\widehat{\mu}}
\newcommand{\wnu}{\widehat{\nu}}
\newcommand{\eps}{\varepsilon}
\newcommand{\Proba}{\mathcal{P}}
\newcommand{\hV}{\widehat{V}}
\newcommand{\wR}{\widetilde{R}}

\usepackage{mathtools}

\DeclareMathOperator{\supp}{supp}

\begin{document}

\title[Phase transition of capacity for the uniform $G_{\delta}$-sets]{Phase transition of capacity for the uniform $G_{\delta}$-sets}


\author[V. Kleptsyn, F. Quintino]{Victor Kleptsyn, Fernando Quintino}
\thanks{The first author was partially supported by the project ANR Gromeov (ANR-19-CE40-0007), as well as 
by the Laboratory of Dynamical Systems and Applications NRU HSE, 
of the Ministry of science and higher education of the RF grant ag. No 075-15-2019-1931. 
The second author was supported by NSF grants DMS-1855541 and DMS-1700143. }

\address{Univ Rennes, CNRS, IRMAR - UMR 6625, F-35000 Rennes, France.}
\email{victor.kleptsyn@univ-rennes1.fr}
\address{Department of Mathematics, University of California, Irvine}
\email{fquintin@uci.edu}

\subjclass[2010]{Primary: 31A15, 31C15.
  Secondary: 28A12.}
\date{\today}
\dedicatory{}
\keywords{Logarithmic capacity, phase transition, parametric Furstenberg theorem.}

\begin{abstract}
We consider a family of dense $G_{\delta}$ subsets of $[0,1]$, defined as intersections of unions of small uniformly distributed intervals, and study their capacity. Changing the speed at which the lengths of generating intervals decrease, we observe a sharp phase transition from full to zero capacity. Such a $G_{\delta}$ set can be considered as a toy model for the set of exceptional energies in the parametric version of the Furstenberg theorem on random matrix products.

Our re-distribution construction can be considered as a generalization of a method applied by Ursell in his construction of a counter-example to a conjecture by Nevanlinna. Also, we propose a simple Cauchy-Schwartz inequality-based  proof of related theorems by Lindeberg and by Erd\"os and Gillis. 
\end{abstract}

\maketitle

\tableofcontents

\section{Introduction}

\subsection{The setting}
Given a compactly supported measure $\mu$ on $\C$, one defines its (Coulomb) \emph{energy} as a double integral:
\begin{equation}\label{eq:I-mu}
I(\mu):=\iint -\log|z-w| \, d\mu(z) d\mu(w).
\end{equation}
The logarithmic capacity of a bounded subset $X\subset \C$ is then defined by minimizing this energy:
\begin{defn}
Let $\Proba(X)$ be the space of probability measures, supported on a (bounded) set $X\subset \C$. 
The \emph{logarithmic capacity} of this set is 
$$
\Ca(X) := \exp(-\inf\{ I(\mu) \mid \mu\in \Proba(X)\}).
$$
\end{defn}
Physicists think of $\mu$ as being a charge distribution on $\C$ and $I(\mu)$ its total energy (see~\cite[pg.~56]{Ra}). There are many tools to measure how thin a set is such as the Lebesgue measure or the Hausdorff dimension. Capacity gauges how far a set is from being a polar set. 


Namely, a \textit{polar} set is traditionally defined (see, for example,~\cite{He}) as a set, on which some subharmonic function~$u$ takes value~$-\infty$.
And it is alternatively defined (\cite[pg.~56]{Ra}) as being of zero capacity, that is, being a subset $E\subset \C$ such 
that $I(\mu)=\infty$ for every non-trivial Borel measure with compact support contained in~$E$.

In most of the literature (\cite{Ra}, \cite[Appendix A]{Si}) this definition is applied to compact subsets of~$\C$. However, it is also 
studied quite extensively for general Borel sets, and this is also the setting in which we will be working in the present paper. 
Our main focus will be the study of ``uniform'' $G_\delta$-sets on the interval~$[0,1]$. That is, given a (sufficiently fast) decreasing sequence $r_n\to 0$, for every $n$ we consider a union of $n$ equally spaced intervals of length~$r_n$:
\begin{equation}\label{eq:V-def}
V_n:= \bigcup_{j=0}^{n-1} J_{j,n}, 
\end{equation}
where $J_{j,n}$ is an open interval of length $r_n$ centered at $c_{j,n} =\frac{j+(1/2)}{n}$:
\begin{align}\label{d.uniform}
J_{j,n}:=(c_{j,n}-\frac{r_n}{2},c_{j,n}+\frac{r_n}{2}), \quad c_{j,n}=\frac{2j+1}{2n}, \quad j=0,1,\dots, n-1.
\end{align}
See Fig.~\ref{fig:G-delta}.
\begin{figure}[!h!]
\includegraphics{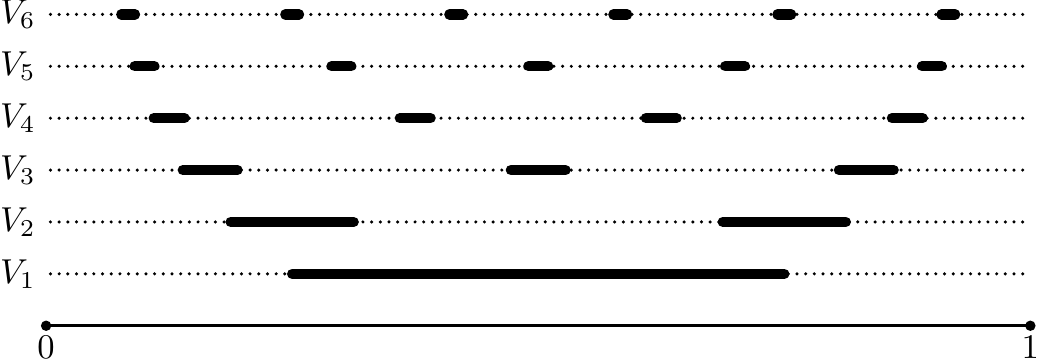}
\caption{Sets $V_n$}\label{fig:G-delta}
\end{figure}

Then we define the \emph{uniform $G_\delta$-set}~$S$, corresponding to the sequence~$r_n$, by
\begin{equation}\label{eq:S-def}
S:=\bigcap_{m=1}^\infty\bigcup_{n=m}^\infty V_n;
\end{equation}
it is immediate to see that~$S$ is indeed a $G_\delta$-subset of~$[0,1]$. 

Our goal is now to study the properties of the set~$S$. Once~$r_n$ goes to~$0$ faster than any power of~$n$, this set is of zero Hausdorff dimension. However, this does not imply anything for its capacity~--- and one can consider the logarithmic capacity as a ``finer'' instrument to describe its properties. 

Such an example is interesting for us for two reasons. First, considering different decrease speed 
for the lengths~$r_n$, we observe a sharp phase transition: while for a fast decrease this set is of zero capacity, for a slower one it turns out to be of \emph{full} capacity (that is, equal to the capacity of~$[0,1]$ itself). Second, such a situation, a $G_{\delta}$-set generated by exponentially small intervals, can be considered as a model case for the set of exceptional energies in the parametric version of the Furstenberg theorem.

In the paper~\cite[Section 1.2]{GK}, the authors have considered the parametric version of a Furstenberg theorem, describing the behaviour to the study of a product 
$$
T_{n,\omega,a}=A_{\omega_n}(a)\dots A_{\omega_1}(a)
$$
of random i.i.d. matrices~$A_{\cdot}(a)\in SL(2,\R)$, depending on a parameter $a$, taking values in some interval $J\subset \R$. 

Under some assumptions, including the individual Furstenberg theorem for every parameter value, it was shown in~\cite[Theorem~1.5]{GK}, that though almost surely for Lebesgue-almost all $a\in J$ one has 
$$
\lim_{n\to\infty} \frac{1}{n} \log \|T_{n,\omega,a}\| = \lambda_F(a)>0,
$$
for the parameters from some  random exceptional subset of parameters $S_e(\omega)$ this equality is violated. Moreover, for the parameters belonging to some (smaller) $G_{\delta}$-set $S_0(\omega)$ one gets
$$
\lim_{n\to\infty} \frac{1}{n} \log \|T_{n,\omega,a}\| =0.
$$

The set $S_e(\omega)$ (and thus $S_0(\omega)$) in~\cite{GK} were shown to have zero Hausdorff dimension. However, the question of their capacity is still open.

Due to their nature, these sets are very similar to those considered in this paper: they are obtained as countable intersection of unions of exponentially small intervals, that are placed in a (more or less) equidistributed way. 
Our theorem thus can be seen as a strong indication for that the exceptional sets of parameters for random matrix products are also of full capacity. 






\subsection{Statement of results}
Recall that the sets $V_n$ in~\eqref{eq:V-def} are unions of~$n$ intervals of length~$r_n$. 
At the moment, we require only $r_n< \frac{1}{n}$ so that the intervals are pairwise disjoint; 
we will discuss possible speeds of decrease for the sequence~$r_n$ later. 

%

Our first result is an easier version of Theorem~\ref{t.Phase.transition}. It is given to demonstrate the technique and part of the proof will be used latter on.

\begin{theorem}[Subexponential uniform $G_\delta$]\label{t.subexponential}
If the sequence $r_n$ decreases subexponentially, then the corresponding uniform $G_\delta$ set~$S$, defined by~\eqref{eq:S-def}, has full capacity. That is, if $|\log r_n|= o(n)$, then
\begin{align*}
\Ca(S)= \Ca([0,1]).
\end{align*}
\end{theorem}
\begin{remark}
As the reader will see, in the proof of this theorem we will not use the fact that \emph{all} the possible denominators $n$ are used in the construction of the set~$S$. Thus, the same conclusion holds for the set $S':=\bigcap_{m=1}^\infty\bigcup_{j=m}^\infty V_{n_j}$, provided that on the subsequence $n_j$ one has $|\log r_{n_j}|= o(n_j)$.
\end{remark}

Theorem \ref{t.subexponential} is already interesting because it shows that there exists a uniform $G_\delta$ set of full capacity. However, its assumption fails  at the decreasing speed that takes place for the random matrices setting, that is exponential. We thus modify it to a more powerful, though more technically complicated, version. This upgraded version is stronger and observe the ``phase transition".

\begin{theorem}[Phase transition]\label{t.Phase.transition}
For $r_n=e^{-n^\alpha}$,
\begin{enumerate}
\item if $\alpha >2$, then $\Ca(S)=0$,
\item if $\alpha <2$, then $\Ca(S)=\Ca([0,1])$.
\end{enumerate}
\end{theorem}

A good question is what happens when $\alpha=2$? We expect that $S$ will still have full capacity, but to establish that, one would have 
to adjust the averaged re-distribution procedure (see Proposition~\ref{p.redistribution.multi.level}), probably making the proof even 
more technical.

It is interesting to note that part~(1) of Theorem \ref{t.Phase.transition} is a partial case of a more general statement, going back to Erd\"os and Gillis~\cite{E-G} and to Lindeberg~\cite{Lin}. Namely, assume that one is given a continuous [concave] increasing function $h$, defined and positive in some right neighborhood of~$0$; they refer to such a function as a \emph{measuring function}. One can then consider the \emph{$h$-volume} of a set $E\subset \R$, defined as 
$$
m_h(E):=\lim_{\eps \to 0+} \, \inf_{\{(x_j,r_j)_{j\in\N}\} \in \mathcal{I}(E,\eps)} \sum_j h(r_j),
$$
where the infimum is taken over the set $\mathcal{I}(E,\eps)$ of covers of $E$ by balls of diameters less than $\eps$:
$$
\mathcal{I}(E,\eps) = \left\{(x_j,r_j)_{j\in \N} \mid \bigcup_j U_{r_j}(x_j)\supset E, \quad \forall j \quad r_j<\eps \right\}.
$$
In particular, the choice $h(r)=r^{\alpha}$ corresponds to the $\alpha$-Hausdorff measure of the set~$E$. They were considering a particular choice of $h_0(r):=\frac{1}{|\log r|}$, and their theorem links the $h_0$-volume (the \emph{logarithmic measure}) to the capacity:
\begin{theorem}[P. Erd\"os, J. Gillis, {\cite[p. 187]{E-G}}, generalizing {Lindeberg~\cite[p.~27]{Lin}}]\label{t:E-G}
If for a set~$E$ one has $m_{h_0}(E)<+\infty$, then $\Ca(E)=0$.
\end{theorem}
This result generalizes the previous one by Lindeberg~\cite[p.~27]{Lin}, where zero capacity was established under the assumption of a \emph{zero} logarithmic measure. An alternate proof of Theorem~\ref{t:E-G} was later provided by L.~Carleson in~\cite[Theorem~2]{Car}.

A particular case of this theorem is obtained by considering a set of the form 
$$
\widetilde{S}=\bigcap_m \bigcup_{k\ge m} I_k,
$$
where $I_k$ are intervals of length $r_k'$. Such a construction includes any uniform $G_{\delta}$ set $S$ 
by enumerating all the intervals $J_{i,n}$ and then adding them one by one instead of by groups of~$V_n$.

It is immediate to notice that if the series $\sum_n \frac{1}{|\log r_n'|}=\sum_n h_0(r_n')$ converges, the $m_{h_0}$-volume of the set $\widetilde{S}$ vanishes, thus implying the following corollary (from which the second part of Theorem~\ref{t.Phase.transition} immediately follows):
\begin{cor}\label{t.alpha>2.zero.cap}
If the series $\sum_n \frac{1}{|\log r_n'|}$ converges, then the set $\widetilde{S}$ is of zero capacity.
\end{cor}

In the same paper~\cite{E-G}, the following conjecture, going back to Nevanlinna's paper~\cite{Nev}, was mentioned: 
\begin{conj}[{\cite{Nev}; see also \cite[(C), p.~186]{E-G}}] 
If for the function $h$ the integral $\int_{0}^{\bullet} \frac{h(t)}{t} \, dt$ diverges and for a closed set $E$ the $h$-volume $m_h(E)$ is finite, then $\Ca(E)=0$.
\end{conj}

In 1937, H.D. Ursell disproved this conjecture, showing that it is false for all functions $h$ except those, for which the conjecture is implied by Theorem~\ref{t:E-G} above. 

The same construction that we use for the proof of Theorem~\ref{t.subexponential} (that can be seen as an extension of Ursell's approach) 
allows to show that for non-closed sets $E$ this conjecture fails even stronger:


\begin{theorem}\label{t:example}
Let $h$ be a measuring function, such that  $\frac{1}{|\log r|}\neq O(h(r))$ as~$r\to 0+$. Then 
there exists a $G_{\delta}$-dense subset $S\subset [0,1]$ with $m_h(S)=0$ and \emph{full} capacity $\Ca(S)=\Ca([0,1])$.
\end{theorem}

The following remark is quite natural, but requires a formal proof, so we put it as a proposition.
\begin{proposition}\label{p:full}
If $X$ is a subset of interval $J$ such that $\Ca(X)=\Ca(J)$, then given any subinterval $J'\subset J$, one has $\Ca(X\cap J')=\Ca(J')$.
\end{proposition}

\begin{cor}
In the same setting as Theorem \ref{t.subexponential} or Theorem \ref{t.Phase.transition} for $\alpha <2$, given any interval $J\subset [0,1]$, we have
\begin{align*}
\Ca(J\cap S)= \Ca(J).
\end{align*}
\end{cor}


\subsection{Plan of the paper}
We start with introducing the re-distribution technique and prove Theorem \ref{t.subexponential} in Section~\ref{s.subexp}; we then apply the same technique to show Theorem~\ref{t:example}. We also prove Proposition~\ref{p:full} in the same section (thus ensuring that ``full capacity'' in inherited by restrictions on the subintervals).

Due to a faster decrease of the intervals, we have to modify the proof of Theorem \ref{t.subexponential}, 
adapting it to the second part of Theorem~\ref{t.Phase.transition}; it is done in Section~\ref{s.phase}. 

Though the statement of Corollary~\ref{t.alpha>2.zero.cap} 
is a particular case of Theorem~\ref{t:E-G} of Lindeberg and Erd\"os and Gillis, 
we note that it can be easily obtained as a corollary of the Cauchy-Schwartz 
inequality. Namely, with help of it one can obtain an upper bound for the capacity of 
a union of intervals; under the assumption of Theorem~\ref{t.alpha>2.zero.cap} 
this bound converges to zero as $m\rightarrow\infty$. 
Moreover, the same argument allows to get another proof of this 
theorem, that is, at the best of our knowledge, not yet known. We thus present this 
(short) proof in Section~\ref{s.zero.capacity}, thus completing the proof of 
Theorem~\ref{t.Phase.transition}.


In the proof of Theorem~\ref{t.subexponential}, there is a tempting shortcut that cannot be taken. If the capacity was continuous for a descending  family of open subsets of $[0,1]$, the arguments of the proof would be much simpler. As we found no examples in the literature demonstrating such non-continuity for open subsets of $[0,1]$, we present such an example in Section~\ref{s.counterexample}.


\section{Subexponential decay}\label{s.subexp}

In this section, we will demonstrate the technique needed to prove Theorem \ref{t.Phase.transition} in a simpler setting by proving Theorem \ref{t.subexponential}. 

Both proofs are based on the idea of \textit{re-distribution}. That is, given a measure $\mu$ that is supported on an interval or on a finite union of intervals, and given a smaller union of intervals $Y\subset X$, we can try finding a new measure $\mu'$, supported on $Y$, close to $\mu$ and with the energy $I(\mu')$ close to~$I(\mu)$. Then Theorem \ref{t.subexponential} will be proven by iterating such a re-distribution on a ``finer" and ``finer" $V_n$'s.

The natural way to do so is to ``move"  the charge, given by the measure~$\mu$, to the closest interval of $Y$, re-distributing it uniformly on each of these intervals; see Fig.~\ref{fig:1}.

\begin{figure}[!h!]
\includegraphics{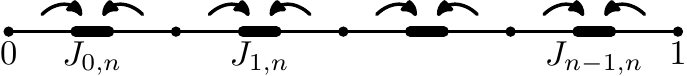}
\caption{The idea of a re-distribution}\label{fig:1}
\end{figure}

However, for ``good'' (absolutely continuous with continuous density) measures~$\mu$ and for the set $Y=V_n$ that is composed of equally spaced intervals of the same lengths, this operation can be approximated by a simpler one, the one of taking the conditional measure. As it is easier to work with, we will proceed with it.


\begin{defn}
Given a finite measure $\mu$ on set $[0,1]$ and measurable set $Y$ with positive measure, we define the \textit{re-distribution} of $\mu$ on $Y$ to be the conditional measure
$$
R(\mu|Y)=\frac{1}{\mu(Y)} \left.\mu \right|_Y.
$$
\end{defn}



\begin{figure}
\includegraphics{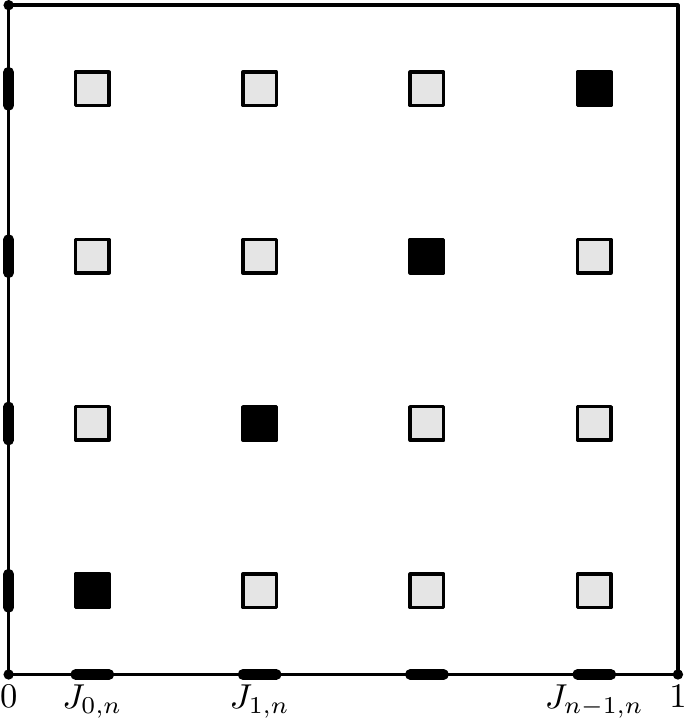}
\caption{Self-interaction (dark squares) and outer-interaction (light ones) parts of the energy integral for the re-distributed measure $R(\mu|V_n)$.}\label{fig:2}
\end{figure}

Now, let $\mu$ be an absolutely continuous measure on $[0,1]$ with continuous density. Let us see how its re-distribution on some $V_n$ changes its energy. The energy of a measure is given by a double integral~\eqref{eq:I-mu}, and the energy of the re-distribution $R(\mu|V_n)$ can be naturally decomposed into two parts: for the variables $x$ and $y$ belonging to the same interval $J_{i,n}$ and to two different ones; see Fig.~\ref{fig:2}. 

It turns out (and this is a statement of Lemma~\ref{l:outer} below) that the second part tends to the initial energy~$I(\mu)$. Meanwhile, the first (``self-interaction'') part behaves as
$$
\frac{|\log r_n|}{n}\cdot \left( \int f^2\,dx +o(1) \right);
$$
see Lemma~\ref{l:self} below. Adding this together, one will get the following proposition.

\begin{proposition}\label{p.redistribution.one.level}
Let $\mu=f(x)\,dx$, where $f\in C([0,1])$, and $\mu_n:=R(f|V_n)$. Then
\begin{equation}\label{eq:I-limit}
I(\mu_n)= I(\mu)+o(1)+\left(\int_0^1 f^2(x) +o(1)\right)\frac{|\log r_n|}{n}.
\end{equation}
\end{proposition}


We postpone its proof until the end of the section, and we will now use it to prove Theorem \ref{t.subexponential}. 
First, note that under the assumptions of this theorem we can omit the self-interaction term:

\begin{cor}\label{c.no.increase}
If $|\log r_n |=o(n),$ then $I(\mu_n)\rightarrow I(\mu)$ as $n\rightarrow \infty$.
\end{cor}

Using it, we immediately get a first full-capacity statement.

\begin{cor}\label{c.union.full.cap}
 If $|\log r_n| = o(n)$, then we have
\begin{align*}
\Ca\left(\bigcup_{n=m}^\infty V_n\right)=\Ca([0,1]) \quad \text{for every }m\in\N.
\end{align*}
\end{cor}

\begin{proof}
Consider the measure $\muo=f_{[0,1]}(x)dx$, where
\[
f_{[0,1]}(x)=\frac{1}{\pi \sqrt{x(1-x)}}.
\]
It is known that this measure minimizes the energy for probability measures supported on $[0,1]$:
\[
I(\muo)=\inf\{ I(\mu) \mid \mu \in\Proba([0,1]) \}, 
\]
and hence that $\Ca ([0,1])= e^{-I(\muo)}$.

Formally, we cannot apply Corollary~\ref{c.no.increase} to this measure, as its density function is not continuous at the endpoints of~$[0,1]$. To avoid this problem, note that there exists a family of probability measures $\mu^{\delta}=f_{\delta}(x)\, dx$ on~$[0,1]$ with $f_{\delta}\in C([0,1])$, such that $I(\mu^{\delta})\to I(\muo)$ as $\delta\to 0$. 

Indeed, consider a family of cut-off densities
$$
\hf_{\delta}(x)=\begin{cases}
\frac{x}{\delta} \cdot f_{[0,1]}(\delta), & x\in [0,\delta),\\
f_{[0,1]}(x), & x\in [\delta,1-\delta],\\
\frac{1-x}{\delta} \cdot f_{[0,1]}(1-\delta), & x\in (1-\delta,1],
\end{cases}
$$
the corresponding (non-probability) measures $\wmu_{\delta}:=\hf_{\delta}(x) \, dx$ on~$[0,1]$, and let 
\[ 
Z_{\delta}:=\wmu_{\delta}([0,1])=\int_0^1 f_{\delta}(x) dx
\] 
be the corresponding normalization constants. Then (for instance, by dominated convergence theorem) we have 
\[
I(\wmu_{\delta}) \to I(\muo),  \quad Z_{\delta} \to 1
\]
as $\delta\to 0$ (here we apply definition~\eqref{eq:I-mu} to non-probability measures~$\wmu_{\delta}$). Hence, for the family 
of probability measures $\mu^{\delta}:=\frac{1}{Z_{\delta}} \wmu_{\delta}$ we also have 
$$
I(\mu^{\delta})= \frac{1}{Z^2_\delta} I(\wmu_{{\delta}})\to I(\muo), \quad \delta\to 0.
$$

Now, let $m\in\N$ be fixed. 
For any $\eps>0$ the above arguments imply that there exists $\delta>0$ such that $I(\mu^{\delta})<I(\muo)+\eps/2$. 
Fix such $\delta>0$ and consider the family of re-distributed measures $\mu^{\delta}_n:=R(\mu^{\delta}|V_n)$.
As the measure~$\mu^{\delta}$ has a continuous density, due to Corollary~\ref{c.no.increase} we have 
$$
I(\mu^{\delta}_n) \to I(\mu^{\delta}), \quad n\to \infty.
$$
In particular, there exists $n\ge m$ such that 
\[
I(\mu^{\delta}_n) \le I(\mu^{\delta}) + \eps/2 \le I(\muo)+\eps.
\] 
As $\eps>0$ was arbitrary, we thus get that 
$$
\inf\{ I(\mu) \mid \mu \in \Proba(\bigcup_{n=m}^{\infty} V_n) \} \le I(\muo),
$$
and hence the desired 
$$
\Ca\left(\bigcup_{n=m}^\infty V_n\right)=\Ca([0,1]) \quad \text{for every }m\in\N.
$$

\end{proof}


It is known that capacity is continuous with respect to any increasing sequence of Borel sets of $\C$ and decreasing sequence of \emph{compact} subsets of~$\C$. Our sequence of sets 
$\left(\bigcup_{n=m}^\infty V_n\right)_{m\in \N}$ is decreasing, but is not closed. 

This is where it would be tempting to conclude by continuity. If the capacity \emph{was} continuous for a decreasing family of open subsets of $[0,1]$, Corollary \ref{c.union.full.cap} would immediately imply Theorem \ref{t.subexponential}.

For decreasing families of (open) subsets of $\C$, it is known that such continuity does not take place; however, all the examples that we found in the literature were essentially two-dimensional. This naturally motivates a question of whether it holds for the subsets of a bounded interval. However, it turns out that it is not the case; we construct a counter-example in Section~\ref{s.counterexample}.

Thus, we continue the proof of Theorem \ref{t.alpha>2.zero.cap} by iterating the re-distributions procedure.
Namely, we have the following

\begin{lemma}\label{l.iter}
Let $|\log r_n|=o(n)$, and $U\subset [0,1]$ be a finite union of intervals, and a measure $\nu=f(x) \, dx$ be a measure with a piecewise-continuous density, supported in $U$. Then for any $\eps>0$ and any $m$ there exist $n\ge m$ and a measure $\nu'$ with a piecewise-continuous density, such that
\begin{align*}
I(\nu')<I(\nu)+\eps,
\end{align*}
and the support of $\nu'$ is contained in $U\cap V_n$.
\end{lemma}


Note that Lemma~\ref{l.iter} suffices to prove Theorem~\ref{t.subexponential}: 
\begin{proof}[Proof of Theorem~\ref{t.subexponential}]
Fix an arbitrary $\eps>0$. We are going to construct a Borel probability measures $\nu_n$, satisfying $I(\nu_n)<I(\muo)+\eps$ and concentrating on the set~$S$. Start (as in the proof of Corollary~\ref{c.no.increase}) with a measure $\nu_0$ with a continuous density on $[0,1]$, satisfying 
$I(\nu_0)<I(\muo)+\frac{\eps}{2}$. 

Recursively applying Lemma~\ref{l.iter}, we construct a sequence $\nu_k$ of measures with a piecewise continuous density, and an increasing sequence of numbers $n_k$, such that the measure $\nu_k$ is supported on $V_{n_1}\cap\dots\cap V_{n_k}$ and that $I(\nu_k)<I(\nu_{k-1}) + \frac{\eps}{2^{k+1}}$.

Then, we have 
$$
I(\nu_k)<I(\muo)+\frac{\eps}{2}+\sum_{j=1}^{k} \frac{\eps}{2^{j+1}} <  I(\muo)+\eps.   
$$
Now, denote $C_k:=\overline{V}_{n_1}\cap \dots \cap \overline{V}_{n_k}$; note that this set differs from the intersection of the corresponding open 
sets~$V_{n_j}$ by at most a finite number of endpoints. 

The family $C_k$ is a decreasing family of compact sets, on which measures $\nu_k$ are respectively supported. 
Hence, any weak limit point $\nu_{\infty}$ of the sequence $\nu_k$ is supported on $C_{\infty}:=\bigcap_k C_k$. 

Recall that passing to the weak limit does not increase the energy (see, e.g., \cite[Lemma~3.3.3]{Ra}). Indeed, for a 
$*$-convergent sequence $\mu_j\to \mu$ of measures on $[0,1]$ one has 
\begin{equation}\label{eq:F-C}
I(\mu) =\lim_{C\to\infty} \int F_C(x,y) \, d\mu(x) \, d\mu(y),
\end{equation}
where $F_C(x,y)=\min (-\log |x-y|,C)$. Thus for any $z<I(\mu)$ there exists~$C$ such that the integral in the 
right hand side of~\eqref{eq:F-C} is at least~$z$. For such~$C$,
\begin{multline*}
\liminf_{j\to\infty} I(\mu_j)\ge \liminf_{j\to\infty} \int F_C(x,y) \, d\mu_j(x) \, d\mu_j(y)  =\\
 = \int F_C(x,y) \, d\mu(x) \, d\mu(y) \ge z,
\end{multline*}
and as $z<I(\mu)$ was arbitrary, we get the desired
$$
\liminf_{j\to\infty} I(\mu_j)\ge I(\mu).
$$
In fact, that is exactly the argument that is used to show the capacity is continuous on decreasing families of compact subsets.

Applying the above argument to our convergent subsequence $\mu_j:=\nu_{k_j}\to \nu_{\infty}$, we get 
$$
I(\nu_{\infty}) \le \lim_j I(\nu_{k_j}) < I(\muo)+\eps.
$$
%
%

On the other hand, $\nu_{\infty}$ is supported on $C_{\infty}\subset S\cup D$, where $D:=\bigcup_k (\partial V_k)$ is a countable set of endpoints. As $I(\nu_{\infty})$ is finite, this measure does not have any atoms hence $\nu_\infty (D)=0$, and thus the measure $\nu_\infty$ is in fact supported on~$S$. Hence, for an arbitrary $\eps>0$ there exists a measure $\nu_{\infty}$, supported on $S$, such that
$$
I(\nu_{\infty})<I(\muo)+\eps,
$$
and thus $\Ca(S)=\Ca([0,1])$.
\end{proof}

Also, still before proving Lemma~\ref{l.iter}, note that the same construction allows to establish Theorem~\ref{t:example}. 
\begin{proof}[Proof of Theorem~\ref{t:example}]
Indeed, assume that the relation $\frac{1}{|\log r|}=O(h(r))$ as $r\to 0+$ does not hold. Then there exists a sequence $r_j\to 0$ along which 
$$
h(r_j) =o(\frac{1}{|\log r_j|}) \, \text{ as } \, j\to \infty.
$$
Extracting a subsequence if necessary, we can assume that 
$$
h(r_j)\cdot |\log r_j|<4^{-j-1}.
$$ 
Choose now integer numbers $n_j=\left[\sqrt{\frac{|\log r_j|}{h(r_j)}} \right]$, roughly speaking, inserting~$n_j$ multiplicatively in the middle between $|\log r_j|$ and $\frac{1}{h(r_j)}$. Then (for all sufficiently large~$j$) we have
\begin{equation}\label{eq:n-j}
n_j h(r_j) < 2^{-j}, \quad \frac{n_j}{|\log r_j|}>2^{j}.
\end{equation}
Consider now the $G_{\delta}$-set
$$
\widetilde{S} = \bigcap_{m=1}^{\infty} \bigcup_{j=m}^{\infty} V_{n_j},
$$
where $V_{n_j}$ are still defined by~\eqref{eq:V-def}--\eqref{d.uniform}; in other words, we are now using only the denominators $n_j$ with the corresponding radii~$r_j$. The first of inequalities in~\eqref{eq:n-j} then implies that this set is of zero $h$-volume, as the series 
$\sum_j n_j h(r_j)$ converges. 
On the other, the second inequality in~\eqref{eq:n-j} ensures that $|\log r_j|= o(n_j)$. Hence the same technique as in the proof of Theorem~\ref{t.subexponential} is applicable, 
showing that the set $\widetilde{S}$ is actually of full capacity on~$[0,1]$.
\end{proof}





\begin{proof}[Proof of Lemma~\ref{l.iter}]
As in the proof of Corollary~\ref{c.union.full.cap}, there exists a family $\nu^{\delta}=f_{\delta}(x) \,dx$ of probability measures, supported on~$U$, 
such that $f_{\delta}\in C([0,1])$ and such that $I(\nu^{\delta})\to I(\nu)$ as $\delta\to 0$. Indeed, if intervals $(a_i,b_i)\subset U$ are the intervals of continuity of the density~$f(x)$, we consider a new (non-probability) density 
$$
\hf_{\delta}(x)=\begin{cases}
\frac{x-a_i}{\delta} \cdot f_{[0,1]}(a_i+\delta), & x\in [a_i,a_i+\delta),\\
f(x), & x\in [a_i+\delta,b_i-\delta],\\
\frac{b_i-x}{\delta} \cdot f_{[0,1]}(b_i-\delta), & x\in (b_i-\delta,b_i];
\end{cases}
$$
see Fig.~\ref{f:joining}. Then, define 
$$
\wnu^{\delta} = \hf_{\delta}(x)\, dx, \quad Z_{\delta}=\wnu^{\delta}([0,1]), \quad \nu^{\delta}=\frac{1}{Z_{\delta}} \wnu^{\delta}.
$$
\begin{figure}[!h!]
\includegraphics{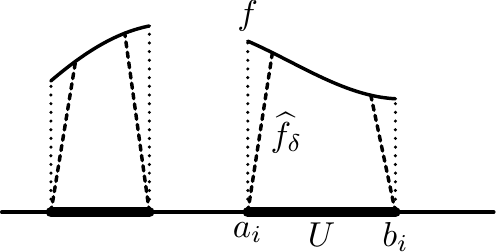}
\caption{Transforming the density~$f(x)$ into a continuous one.}\label{f:joining}
\end{figure}

As before, we get 
$$
Z_{\delta}\to 1, \quad I(\wnu^{\delta})\to I(\nu) \quad \text{as} \quad \delta\rightarrow 0,
$$
and hence $I(\nu^{\delta})= \frac{1}{Z_{\delta}^2} I(\wnu^{\delta}) \to I(\nu)$.

Now, if $\eps>0$ is given, take such a measure $\nu^{\delta}$ that \mbox{$I(\nu^{\delta})<I(\nu)+\frac{\eps}{2}$}. Applying Proposition~\ref{p.redistribution.one.level} to the re-distributions $\nu^{\delta}_n:=R(\nu^{\delta}|V_n)$ of this measure, we get that 
$I(\nu^{\delta}_n)=I(\nu^{\delta})+o(1)$. Hence, for some $n\ge m$ we have 
$$
I(\nu^{\delta}_n) < I(\nu^{\delta})+\frac{\eps}{2} < I(\nu)+\eps;
$$
by construction, the measure $\nu^{\delta}_n$ is supported on $V_n\cap U$.
\end{proof}


\begin{proof}[Proof of Proposition~\ref{p.redistribution.one.level}]
We conclude the section with the proof of Proposition \ref{p.redistribution.one.level}. First, note that the normalization constant $\mu(V_n)$ satisfies
$$
\mu(V_n)=nr_n \cdot (1+o(1)).
$$
Indeed, for any $\eps>0$ due to the uniform continuity of $f(x)$ for all sufficiently large $n$ we have $|f(x)-f(c_{i,n})| <\eps$ for all $x\in J_{i,n}$. Hence,
$$
\left| \int_{J_{i,n}} f(x) \, dx - f(c_{i,n}) r_n \right| <\eps r_n;
$$
summing over $i=0,\dots,n-1$ and dividing by $nr_n$, we get 
$$
\left|\frac{1}{n r_n} \mu(V_n) - \frac{1}{n} \sum_{i=0}^{n-1} f(c_{i,n}) \right| < \eps.
$$
Now, $\frac{1}{n} \sum_{i=0}^{n-1} f(c_{i,n}) \to \int_{[0,1]} f(x) \, dx=1$; as $\eps>0$ was arbitrary, we thus get the desired 
$$
\frac{1}{n r_n} \mu(V_n) = 1 +o(1).
$$

Now, multiplying~\eqref{eq:I-limit} by $(1+o(1))$ does not change its right hand side, so we can consider (non-probability) measure $\frac{1}{n r_n} \mu|_{V_n}$ instead of~$R(\mu|V_n)=\frac{1}{\mu(V_n)}\mu|_{V_n}$. It is also useful to consider extend the definition of the energy, considering it as a bilinear form: for any two (not necessarily probability) measures $\mu, \nu$ let
\begin{align*}
I(\nu,\mu)=-\iint \log|x-y|\, d\nu(x)\,d\mu(y).
\end{align*}
It is immediate to note that
\begin{enumerate}
\item $I(\nu)=I(\nu,\nu)$,
\item $I(\nu,\mu)=I(\mu,\nu),$
\item $I(\nu,\mu)>0$, if $\mu$ and $\nu$ are supported on~$[0,1]$,
\item $I(\nu,\mu+\mu')=I(\nu,\mu)+I(\nu,\mu')$; $I(\nu, c\mu)=c I(\nu,\mu)$.
\end{enumerate}

The measure $\frac{1}{n r_n} \mu|_{V_n}$ can be written as 
$$
\frac{1}{n r_n} \mu|_{V_n} = \frac{1}{n} \sum_{i=0}^{n-1} \mu_{i,n}, 
$$
where 
$\mu_{i,n}:=\frac{1}{r_n} \mu|_{J_{i,n}}$. Thus, we can decompose $I(\frac{1}{n r_n} \mu|_{V_n})$ as 
\begin{multline*}
I(\frac{1}{n r_n} \mu|_{V_n}) = \frac{1}{n^2} \sum_{i,j=0}^{n-1} I(\mu_{i,n}, \mu_{j,n}) = \\ 
= \frac{1}{n^2} \sum_i I(\mu_{i,n}) + \frac{1}{n^2} \sum_{i\neq j} I(\mu_{i,n},\mu_{j,n}).
\end{multline*}

Proposition~\ref{p.redistribution.one.level} now follows from the next two Lemmas,~\ref{l:self} 
and~\ref{l:outer}, estimating the diagonal and off-diagonal sums respectively.
\end{proof}



\begin{lemma}\label{l:self}
\begin{equation}\label{eq:self}
\frac{1}{n^2}\sum_{i=0}^{n-1} I (\mu_{i,n}) = \frac{|\log r_n|}{n} \left( \int_0^1 f^2(x) dx + o(1) \right).
\end{equation}
\end{lemma}

\begin{lemma}\label{l:outer}
\begin{equation}\label{eq:outer}
\frac{1}{n^2}\sum_{i\neq j} I (\mu_{i,n},\mu_{j,n}) = I(\mu) + o(1).
\end{equation}
\end{lemma}

\begin{proof}[Proof of Lemma~\ref{l:self}]
Let us first estimate $I(\mu_{i,n})$ for an individual $i$, comparing it with the energy of the uniform measure $\frac{1}{r_n} dx|_{J_{i,n}}$. Indeed, 
$$
I(\mu_{i,n}) = \iint_{J_{i,n}} (-\log|x-y|) f(x) f(y) \, \frac{dx}{r_n} \frac{dy}{r_n},
$$
and hence
\begin{align}\label{e.outer1}
(\min_{J_{i,n}} f(x))^2 \cdot I(\frac{1}{r_n} dx|_{J_{i,n}}) \le I(\mu_{i,n}) \le  (\max_{J_{i,n}} f(x))^2 \cdot I(\frac{1}{r_n} dx|_{J_{i,n}}).
\end{align}
Rescaling and a change of variables immediately shows that 
\begin{align}\label{e.outer2}
I(\frac{1}{r_n} dx|_{J_{i,n}}) = \log r_n + I(\left. dx\right|_{[0,1]}) = \log r_n \cdot (1+o(1)).
\end{align}
Fix an arbitrarily small $\eps>0$; for all sufficiently large $n$, the function $f(x)^2$ then oscillates less than $\eps/2$ on any of the intervals $J_{i,n}$. Joining it with \eqref{e.outer1} and \eqref{e.outer2}, for all sufficiently large $n$ we get
$$
\frac{1}{|\log r_n|} I(\mu_{i,n}) \in (f^2(c_{i,n})-\eps,f^2(c_{i,n})+\eps).
$$
Summing over $i$ and dividing by $n$, we get 
$$
\left| \frac{1}{n |\log r_n|} \sum_i I(\mu_{i,n}) - \frac{1}{n} \sum_i f^2(c_{i,n}) \right| <\eps.
$$
The second sum converges to the Riemann integral $\int_0^1 f^2(x) dx$; as $\eps>0$ was arbitrary, we get 
$$
\frac{1}{n |\log r_n|} \sum_i I(\mu_{i,n}) = \int_0^1 f^2(x) dx + o(1).
$$
Multiplying by $\frac{|\log r_n|}{n}$, we get the desired~\eqref{eq:self}.
\end{proof}

Before proceeding with Lemma~\ref{l:outer}, let us estimate the interaction energy for uniformly distributed measures 
on the subintervals, comparing it to the interaction energy between point charges at their centers.

\begin{figure}[!h!]
\includegraphics{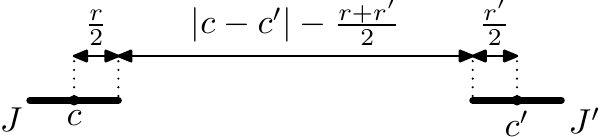}
\caption{Two intervals $J,J'$ and their centers.}\label{fig:intervals}
\end{figure}

\begin{lemma}
Let $J,J'\subset [0,1]$ be two disjoint intervals with centers $c,c'$ and with lengths $r,r'$ respectively (see Fig.~\ref{fig:intervals}). Then the interaction energy between the uniform measures on these intervals satisfies 
$$
-\log |c-c'| < I(\frac{1}{r} dx|_J, \frac{1}{r'} dx|_{J'}) < (-\log |c-c'|) + \Delta, 
$$
where $\Delta=\min (2, (-\log (1-\frac{r+r'}{2|c-c'|})))$.
\end{lemma}

\begin{proof}
The lower bound is implied by the Jensen's inequality: as the function $F(x,y)=-\log|x-y|$ is convex on the rectangle $J\times J'$,
$$
\iint_{J\times J'} F(x,y) \, \frac{dx}{r} \, \frac{dy}{r'} > F(c,c')= -\log |c-c'|.
$$
Now, for any $x\in J, \, y\in J'$ we have 
\begin{multline*}
- \log |x-y| = -\log |c-c'| - \log \frac{|x-y|}{|c-c'|} \\ =-\log |c-c'| - \log \left(1- \frac{|c-c'|-|x-y|}{|c-c'|}\right)
\end{multline*}
and the upper bound by $(-\log (1-\frac{r+r'}{2|c-c'|}))$ follows as it is the maximal possible value of the second term.

To get a uniform upper bound by~$2$, consider first the interaction between a uniform measure and a point charge. 
Note that for any $y\in J'$ we have 
\begin{multline*}
\int_J (-\log|x-y|) \frac{dx}{r} = -\log |c-y| - \frac{|c-y|}{r} \cdot \int_{-\frac{r/2}{|c-y|}}^{\frac{r/2}{|c-y|}} \log (1+s) \, ds = 
\\ 
= -\log |c-y| - \frac{|c-y|}{r} \cdot \int_{0}^{\frac{r/2}{|c-y|}} \log (1-s^2) \, ds ;
\end{multline*}
as the function $-\log(1-s^2)$ is monotone increasing, the maximal value of its average will be if it is averaged on the largest possible 
interval, that is, over $[0,1]$ (that corresponds to $|c-y|=r/2$, in other words, $y$ being on the boundary of $J$). In this case, a straightforward computation shows that the second term is equal to
$$
 \int_{-1}^1  (-\log (1+s)) \, \frac{ds}{2} = 1- \log 2<1.
$$
Thus, for any $y\in J'$ we have 
$$
\int_J (-\log|x-y|) \frac{dx}{r} < -\log |c-y| +1.
$$
Finally, averaging with respect to $y\in J'$, we get 
$$
\iint_{J\times J'} (-\log|x-y|) \frac{dx}{r} \frac{dy}{r'} < \int_{J'} (-\log |c-y|)\, \frac{dy}{r'} +1 < \log|c-c'| +2.
$$
\end{proof}


%

\begin{proof}[Proof of Lemma~\ref{l:outer}]
Fix an arbitrary small $\delta>0$, and let $M:=\max_{[0,1]} f(x)$. Let us decompose the sum in the left hand side of~\eqref{eq:outer} into two parts, depending on whether the centers $c_{i,n}$ and $c_{j,n}$ are closer than $\delta$ to each other:
$$
\frac{1}{n^2}\sum_{i\neq j} I (\mu_{i,n},\mu_{j,n}) = 
\frac{1}{n^2}\sum_{0<|c_{i,n}- c_{j,n}|<\delta} I (\mu_{i,n},\mu_{j,n}) + 
\frac{1}{n^2}\sum_{|c_{i,n}- c_{j,n}|\ge \delta} I (\mu_{i,n},\mu_{j,n}).
$$
Note that the first sum can be bounded by an arbitrarily small constant by choosing an appropriate $\delta>0$. Indeed, note first that 
$$
I(\mu_{i,n},\mu_{j,n})< M^2 I(\frac{1}{r_n} dx|_{J_{i,n}},\frac{1}{r_n} dx|_{J_{j,n}}).
$$
Taking $\delta<1/e^2$ and thus ensuring $-\log|c_{i,n}-c_{j,n}|>2$ once $|c_{i,n}-c_{j,n}|<\delta$, we get 
\begin{multline*}
\frac{1}{n^2}\sum_{0<|c_{i,n}- c_{j,n}|<\delta} I (\mu_{i,n},\mu_{j,n}) < \frac{1}{n^2}M^2 \sum_{0<|c_{i,n}- c_{j,n}|<\delta} I(\frac{1}{r_n} dx|_{J_{i,n}},\frac{1}{r_n} dx|_{J_{j,n}}) \\ 
< 2\frac{M^2}{n^2} \sum_{0<|c_{i,n}- c_{j,n}|<\delta} \left( -\log |c_{i,n}-c_{j,n}| \right) 
\end{multline*}
Now, for each $i$ we have 
\begin{equation}\label{eq:upper-log}
\frac{1}{n} \sum_{j: \atop  0<|c_{i,n}- c_{j,n}|<\delta} ( -\log |c_{i,n}-c_{j,n}| ) \leq \frac{2}{n} \sum_{k=1}^{[\delta n]} (-\log \frac{k}{n})  
<  2 \int_0^{\delta} (-\log s) \, ds,
\end{equation}
as the function $(-\log s)$ is decreasing on~$[0,1]$; see Fig.~\ref{fig:log}, left. Averaging~\eqref{eq:upper-log} over $i$, we get 
$$
\frac{1}{n^2}\sum_{0<|c_{i,n}- c_{j,n}|<\delta} I (\mu_{i,n},\mu_{j,n}) < 4 M^2  \int_0^{\delta} (-\log s) \, ds.
$$

\begin{figure}[!h!]
\includegraphics{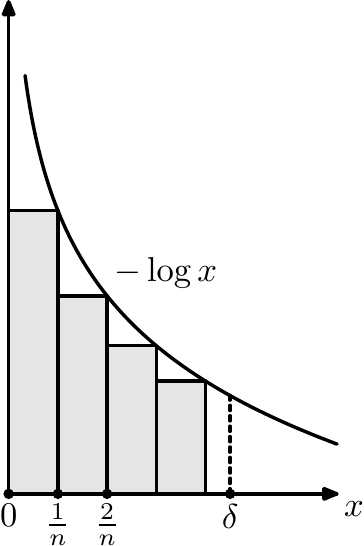} \qquad 
\includegraphics{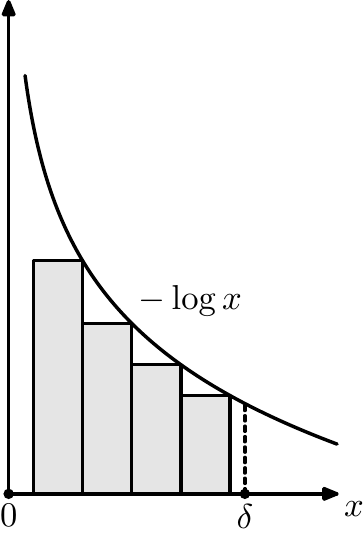}
\caption{Comparing integral sums and the integral for the $-\log x$ function: nonshifted (left) and shifted (right) sums.}\label{fig:log}
\end{figure}

As the integral in the right hand side tends to $0$ as $\delta\to 0$, for any $\eps>0$ we have
\begin{equation}\label{eq:delta-0}
\exists \delta_0>0 : \quad \forall \delta<\delta_0 \, \forall n\in \N \quad \frac{1}{n^2}\sum_{0<|c_{i,n}- c_{j,n}|<\delta} I (\mu_{i,n},\mu_{j,n}) <\eps.
\end{equation}

Now, for any fixed $\delta>0$, the function $f(x) f(y) (-\log|x-y|)$ is uniformly continuous on the subset $\{|x-y|\ge \delta\}$, and hence 
\begin{equation}\label{eq:i-limit}
\frac{1}{n^2}\sum_{|c_{i,n}- c_{j,n}|\ge \delta} I (\mu_{i,n},\mu_{j,n}) \xrightarrow[n\to\infty]{} \int_{\{|x-y|\ge \delta\}}f(x) f(y) (-\log|x-y|) \, dx\, dy.
\end{equation}
The integral in the right hand side of~\eqref{eq:i-limit} tends to $I(\mu)$ as $\delta\to 0$. Hence, for any sufficiently small $\delta$ it is $\eps$-close to $I(\mu)$. Fixing such $\delta<\delta_0$, from~\eqref{eq:i-limit} for all sufficiently large $n$ we get 
$$
\left| \frac{1}{n^2}\sum_{|c_{i,n}- c_{j,n}|\ge \delta} I (\mu_{i,n},\mu_{j,n}) - I(\mu) \right| < 2\eps,
$$
and joining it with~\eqref{eq:delta-0},
$$
\left| \frac{1}{n^2}\sum_{i\neq j} I (\mu_{i,n},\mu_{j,n}) - I(\mu) \right| < 3\eps.
$$
As $\eps>0$ was arbitrary, we get the desired 
$$
\frac{1}{n^2}\sum_{i\neq j} I (\mu_{i,n},\mu_{j,n}) = I(\mu)+o(1).
$$
This completes the proof of Lemma~\ref{l:outer}, and hence of Proposition~\ref{p.redistribution.one.level}.
\end{proof}


\begin{proof}[Proof of Proposition~\ref{p:full}]
For any interval $[a,b]$, denote $\mu_{[a,b]}$ to be the probability measure with the least energy on this interval, that is,
$$
\mu_{[a,b]}=\rho_{[a,b]}(x)\, dx, \quad \rho_{[a,b]}(x)=\frac{1}{\pi\sqrt{(x-a)(b-x)}}.
$$

By assumption of the full capacity, there exists a sequence of measures $\nu_n$, supported on $X$, such that $I(\nu_n)\to I(\mu_J)$. Upon extracting a subsequence, we can assume that this sequence of measures converges weakly. Again using the fact that passing to the weak limit 
does not increase the energy, we get 
\begin{equation}\label{eq:I-lim}
I(\lim_{n\to\infty} \nu_n) \le \lim_{n\to\infty} I(\nu_n)=I(\mu_J);
\end{equation}
as $\mu_J$ is the unique minimum of the energy function on $\Proba(J)$, we thus have $\nu_n\to\mu_J$ as $n\to\infty$. Moreover, the inequality in~\eqref{eq:I-lim} turns into an equality. And an equality in~\eqref{eq:I-lim} is equivalent to the uniform integrability of the function $-\log|x-y|$ w.r.t. these measures, that is, to
$$
\forall \eps>0 \quad \exists r>0: \quad \forall n \quad \iint_{|x-y|<r} \, \big|\log |x-y|\big| \, d\nu_n(x) \, d\nu_n(y) < \eps.
$$
(If it does take place for some $\eps>0$, the sides of the inequality in~\eqref{eq:I-lim} differ by at least~$\eps$, and vice versa.)


Now, for every $\delta>0$, take a continuous positive function $f_{\delta}\in C(J)$, supported on~$J'$, such that the measures $f_{\delta} \, dx|_{J'}$ are probability ones and converge to $\mu_{J'}$, and so do their energies: 
\begin{equation}\label{eq:I-f-delta}
I(f_{\delta} \, dx|_{J'}) \to I(\mu_{J'});
\end{equation}
it can be done in the same way as the cut-off is done on the first step of the proof of Corollary~\ref{c.union.full.cap}. These measures can then be re-written as 
$$
f_{\delta}(x)\, dx|_{J'} =\frac{f_{\delta}(x)}{\rho_J(x)} \, \rho_J(x) \, dx = \frac{f_{\delta}(x)}{\rho_J(x)} \, \mu_J;
$$
denote then $\widetilde{f}_{\delta}(x):=\frac{f_{\delta}(x)}{\rho_J(x)}$.

Consider the measures 
$$
\widehat{\mu}_{\delta,n}:= \widetilde{f}_{\delta}(x) \nu_n,
$$
and their normalized versions 
$$
\mu_{\delta,n}=\frac{1}{Z_{\delta, n}}\widehat{\mu}_{\delta,n}, \quad Z_{\delta, n}:=\widehat{\mu}_{\delta,n}(J).
$$
For each $\delta$, the measures $\widehat{\mu}_{\delta,n}$ converge weakly as $n\to \infty$ to $\widetilde{f}_{\delta}(x)\mu_{J}=f_{\delta}(x) \, dx|_{J'}$; as the limit measure is a probability one, we have 
$$
Z_{\delta,n}=\int \widetilde{f}_{\delta}(x) d\nu_n(x)  \xrightarrow[n\to \infty]{}  \int \widetilde{f}_{\delta}(x) d\mu_{J} = 1.
$$ 
Now, as the function $\widetilde{f}_{\delta}$ is bounded, the function $-\log|x-y|$ is still uniformly integrable w.r.t. these measures, and hence 
$$
I(\widehat{\mu}_{\delta,n}) \xrightarrow[n\to\infty]{} I(f_{\delta} \, dx|_{J'}). 
$$
Thus, we also have 
$$
I(\mu_{\delta,n}) = \frac{1}{Z_{\delta,n}^2} I(\widehat{\mu}_{\delta,n}) \xrightarrow[n\to\infty]{} I(f_{\delta} \, dx|_{J'})
$$
Now, passing to the limit as~$\delta\to 0$ and using~\eqref{eq:I-f-delta}, we get 
$$
\lim_{\delta\to 0} \lim_{n\to\infty} I(\mu_{\delta,n}) = I(\mu_{J'}).
$$
As the measures $\mu_{\delta,n}$ are supported on $X\cap J'$, and $\mu_{J'}$ is the least energy probability 
measure on~$J'$, we get the desired 
$$
\Ca(X\cap J')=\Ca(J').
$$
\end{proof}


\section{Phase transition}\label{s.phase}

Let us move on to prove Theorem \ref{t.Phase.transition}. The key ingredient in the sub-exponential case was that the re-distribution of $\mu_n$ on a single level, $V_n$, of a given measure $\mu$ gave us a close approximation of $I(\mu)$. If $r_n=e^{-n^\alpha}$, then Proposition \ref{p.redistribution.one.level} yields 
\begin{align*}
I(\mu_n)=I(\mu)+o(1)+\left(\int_0^1 f^2(x) +o(1) \right)n^{\alpha-1}.
\end{align*}
For $1\leq\alpha<2,$ a simple re-distribution does not suffice, as the self-interaction term has an asymptotics of $n^{\alpha-1}$ and hence does not tend to zero. The re-distribution thus will have to be done on multi-levels. Namely, let
\begin{align*}
F_m:=\{n=m,\ldots,2m-1\,:\, n\text{ is prime}\},
\end{align*}
that is, the set of prime numbers in $[m,2m-1]$, and denote by $N_m=\#F_m$ its cardinality.

Notice that $V_p$ and $V_q$ are disjoint for distinct $p,q\in F_m$. Indeed, this follows from the fact that the centers $c_{k,p}=\frac{2k+1}{2p}$ are distinct for $p\in F_m$, and that 
\begin{align*}
\left|\frac{a}{2p}-\frac{b}{2q}\right|=&\left|\frac{aq-bp}{2pq}\right|\geq\frac{1}{2m^2}> e^{-m^\alpha}.
\end{align*}

Let $\mu_n$ be the re-distribution of $\mu$ on $V_n$, where $n\in F_m$. Given a collection of positive numbers $\{p_n\}_{n\in F_m}$ such that 
\begin{align*}
\sum_{n\in F_m} p_n=1,
\end{align*}
consider a \emph{averaged re-distribution}: 
$$\mu^m=\wR_m(\mu):=\sum_{n\in F_m} p_n\mu_n,$$
that is a convex combination of measures $\mu_n$, supported on a finite union
$$
\hV_m:=\bigcup_{n\in F_m} V_n.
$$

The averaging allows to regain control on the self-interaction term. That is, the energy of the averaged measure $\mu^m$ satisfies 
\begin{equation}\label{eq:I-averaged}
I(\mu^m) = \sum_{n\in F_m} p_n^2 I(\mu_n) + \sum_{i\neq j} p_i p_j I(\mu_i,\mu_j).
\end{equation}
Take $p_i$ to be uniform: let $p_i=\frac{1}{N_m}$ for every $i\in F_m$. We have $I(\mu_n)=O(n^{\alpha-1})$, and due to the Prime Number Theorem $N_m \sim \frac{m}{\log m}$ as $m\to\infty$. Hence, the first term in \eqref{eq:I-averaged} can be estimated as
\begin{multline}\label{eq:averaged-self}
\sum_{n\in F_m} p_n^2 I(\mu_n) = \frac{1}{N_m^2} \sum_{n\in F_m} I(\mu_n) \le \frac{1}{N_m} \max_{n\in F_m} I(\mu_n)
\\
= \frac{O(m^{\alpha-1})}{m/\log m} =O\left(\frac{\log m}{m^{2-\alpha }}\right) = o(1),
\end{multline}
as $\alpha<2$.




On the other hand, we claim that the interaction energy between different $\mu_n$'s is close to the one of the initial measure $\mu$:
\begin{lemma}\label{l.otherinteraction.phase}
Let $\mu=f(x)\, dx$ be a measure with a continuous density on~$[0,1]$. Then for $n,n'\in F_m$, $n\neq n'$ we have  
\begin{align*}
I(\mu_n,\mu_{n'})= I(\mu)+o(1)
\end{align*}
(uniformly on the choice of $n$ and $n'$) as $m\rightarrow \infty$.
\end{lemma}

Postponing its proof till the end of this section, note that it immediately imples
\begin{proposition}\label{p.redistribution.multi.level}
Let $\mu=f(x)\, dx$ be a measure with a continuous density on~$[0,1]$. Then for the family of its averaged re-distributions~$\mu_m=\wR_m(\mu)$ we have 
\begin{align*}
I(\mu^m) = I(\mu)+o(1).
\end{align*}
\end{proposition}
\begin{proof}
Due to~\eqref{eq:I-averaged}, the energy $I(\mu^m)$ is the sum of two terms; the first one is~$o(1)$ due to~\eqref{eq:averaged-self}, while the second is $I(\mu)+o(1)$ due to Lemma~\ref{l.otherinteraction.phase}.
\end{proof}

We then get 

\begin{lemma}\label{l.multi.level}
Let $r_n=e^{-n^{\alpha}}$, where $\alpha<2$. Let $U\subset [0,1]$ be a finite union of intervals, and a measure $\nu=f(x) \, dx$ be a measure with a piecewise-continuous density, supported in $U$. Then for any $\eps>0$ and any $k$ there exist $m\ge k$ and a measure $\nu'$ with a piecewise-continuous density, such that
\begin{align*}
I(\nu')<I(\nu)+\eps,
\end{align*}
and the support of $\nu'$ is contained in $U\cap \hV_m$.
\end{lemma}
\begin{proof}
As in the proof of Lemma~\ref{l.iter}, we can find a measure $\nu_{\delta}=f_{\delta}(x)\, dx$ with continuous density on $[0,1]$, such that $\supp \nu_{\delta}\subset \supp \nu$ and that $I(\nu_{\delta})<I(\nu)+\frac{\eps}{2}$. Applying Proposition~\ref{p.redistribution.multi.level} to $\mu=\nu_{\delta}$ concludes the proof.
\end{proof}

\begin{proof}[Proof of Theorem~\ref{t.Phase.transition}]
We now deduce Theorem~\ref{t.Phase.transition} from Lemma~\ref{l.multi.level} in exactly the same way, as earlier we have deduced Theorem~\ref{t.subexponential} from Lemma~\ref{l.iter}. Namely, for any $\eps>0$ we iterate the re-distribution procedure, obtaining a family of measures $\nu_k$ with continuous density on~$[0,1]$, for which we control both the supports and the energy. 

To do so, we start with the measure $\nu_0$ that is supported on~$[0,1]$ and that satisfies $I(\nu_0)<I(\muo)+\frac{\eps}{2}$. Now, if a measure $\nu_{k-1}$ is already constructed, due to Lemma~\ref{l.multi.level} there exists a measure $\nu_{k}$ with 
$$
I(\nu_{k})<I(\nu_{k-1})+\frac{\eps}{2^{k+1}} \, \text{ and } \, \supp \nu_{k}\subset \supp \nu_{k-1} \cap \hV_{m_k}
$$
for some $m_k>k$. Any accumulation point $\nu_{\infty}$ of the measures $\nu_k$ is thus supported on a intersection of closures 
$$
\bigcap_k  cl \left( \hV_{m_k} \right) \subset S \cup D,
$$
where $D$ is a countable set of endpoints of $V_n$'s, and satisfies 
$$
I(\nu_{\infty}) < (I(\muo)+\frac{\eps}{2}) + \sum_{k=1}^{\infty} \frac{\eps}{2^{k+1}} = I(\muo)+\eps.
$$
As a finite energy measure, the measure $\nu_{\infty}$ does not charge a countable set $D$, and is thus supported on~$S$. As $\eps>0$ was arbitrary, we thus get 
$$
\inf_{\nu\in \Proba(S)} I(\nu) =I(\muo),
$$
and hence the desired $\Ca(S)=\Ca([0,1])$.
\end{proof}

We conclude this section with the proof of Lemma~\ref{l.otherinteraction.phase}.
\begin{proof}[Proof of Lemma~\ref{l.otherinteraction.phase}]
As in the proof of Lemma~\ref{l:outer}, fix an arbitrarily small $\delta>0$, and decompose 
$$
I(\mu_n,\mu_{n'})= \frac{1}{n\cdot n'}\sum_{i=0}^{n-1} \sum_{j=0}^{n'-1} I(\mu_{i,n},\mu_{j,n'})
$$
into two parts, depending on the distance $|c_{i,n}-c_{j,n'}|$:
\begin{equation}\label{eq:I-n-np}
I(\mu_n,\mu_{n'})= \frac{1}{nn'}\sum_{|c_{i,n}-c_{j,n'}|<\delta} I(\mu_{i,n},\mu_{j,n'}) + \frac{1}{nn'}\sum_{|c_{i,n}-c_{j,n'}|\ge \delta} I(\mu_{i,n},\mu_{j,n'}).
\end{equation}
The sum over intervals whose centers are closer than $\delta$ from each other can be made arbitrarily small by a choice of~$\delta$ and by taking sufficiently large~$m$. Indeed, for any fixed $j$ we have 
\begin{multline}\label{eq:delta-indiv}
\frac{1}{n}\sum_{i: \, |c_{i,n}-c_{j,n'}|<\delta} (-\log |c_{i,n}-c_{j,n'}|) < \\ < -\frac{2}{n} \log \min_{i} |c_{i,n}-c_{j,n'}| + 2\int_0^{\delta} (-\log s) \, ds,
\end{multline}
see Fig.~\ref{fig:log}, right. Due to the estimates above the minimal distance $\min_{j} |c_{i,n}-c_{j,n'}|$ is at least $\frac{1}{2m^2}$, so the first summand does not exceed $\frac{2}{m} \log (2m^2)$ and hence tends to~$0$. The second can be made arbitrarily small due to the integrability of the function~$\log$ at~$0$. Finally, averaging~\eqref{eq:delta-indiv} over~$j$, we get the desired (arbitrarily small) bound for the first summand in~\eqref{eq:I-n-np}.

On the other hand, for any fixed $\delta$, the function $f(x) f(y) (-\log |x-y|)$ is continuous on the set $|x-y|\ge\delta$, 
and the second summand in~\eqref{eq:I-n-np} behaves like its Riemann sum. Hence, we have 
$$
\frac{1}{nn'}\cdot \sum_{|c_{i,n}-c_{j,n'}|\ge \delta} I(\mu_{i,n},\mu_{j,n'}) \to \iint_{|x-y|\ge\delta}f(x) f(y) (-\log |x-y|) \, dx \, dy
$$
uniformly in $n,n'\in F_m$ as $m\to \infty$.

For any $\eps>0$, take $\delta$ sufficiently small so that the integral in the right hand side is $\frac{\eps}{2}$-close to $I(\mu)$, and that the first summand in~\eqref{eq:I-n-np} does not exceed $\frac{\eps}{2}$ for all sufficiently large~$m$. Then, we have 
$$
\left| I(\mu_n,\mu_{n'}) - I(\mu) \right| < \frac{\eps}{2} + \frac{\eps}{2} = \eps, 
$$
and as $\eps>0$ is arbitrary, this concludes the proof.
\end{proof}





\section{Zero capacity: Lindeberg and Erd\"os-Gillis theorems}\label{s.zero.capacity}


This section is devoted to the Cauchy-Schwartz based proof of Corollary~\ref{t.alpha>2.zero.cap}, as well as of Lindeberg and  
Erd\"os-Gillis' Theorem~\ref{t:E-G}, and hence of the first part of 
Theorem~\ref{t.Phase.transition}. The key step is the following.

\begin{lemma}\label{l:I-ints} 
Let $J_1',J_2,\ldots \subset [0,1]$ be a sequence of intervals of length $|J_k'|=:r'_k$, such that the series $\sum_{k=1}^\infty \frac{1}{|\log r_k|}$ converges. Then
\begin{align*}
I(\mu)\geq \frac{1}{\sum_{k=m}^\infty 1/|\log r_k'| }.
\end{align*}

\end{lemma}

\begin{proof}
We transform the union $\bigcup_{k=1}^\infty J_k'$ into a disjoint one by setting
\begin{align*}
\tilde{V_1}:=J_1', \quad \tilde{V_k}:=J_k'\setminus \bigcup_{i=1}^{k-1}J_i'.
\end{align*}
Let $\mu$ be any measure supported on $\bigcup_{k=1}^\infty J_k'$; denote $p_k:=\mu(\tilde{V_k})$. Then $\sum_{k}p_k= \mu( \bigcup_{k=1}^\infty J_k')=1.$ Without loss of generality, we can assume $p_k>0$ for all $k$, otherwise removing the corresponding $J_k'.$ 

Let $\mu_k':= \frac{1}{p_k}\left. \mu \right|_{\widetilde{V}_k}$ be the corresponding conditional measures. Then, 
\begin{align*}
\mu=\sum_k p_k\mu_k',
\end{align*}
and thus
\begin{align*}
I(\mu)=\sum_{k,l} p_kp_l I(\mu_k',\mu_l')\geq \sum_k p_k^2 I(\mu_k').
\end{align*}
Now, the measures $\mu_k$ is supported on $J_k'$, that is an interval of length $r_k'$, and hence $I(\mu_k')\geq |\log r_k'|.$ Thus, 
\begin{align*}
I(\mu)\geq \sum_{k}p_k^2|\log r_k'|.
\end{align*}
Applying Cauchy-Schwartz inequality, we get 
\begin{multline*}
\left( \sum_k p_k^2 |\log {r_k'}|\right)\left( \sum_k \frac{1}{|\log r_k'|}\right)\ge
\\
\ge \left( \sum _k \sqrt{p_k^2|\log r_k'|\cdot \frac{1}{|\log r'_k|}} \right)^2= \left( \sum_k p_k\right)^2=1,
\end{multline*}
and hence
\begin{equation}\label{eq:I-bound}
I(\mu) \geq \sum_k p_k^2 |\log r_k' |\geq \frac{1}{\sum_k \frac{1}{|\log r_k'|}}
\end{equation}

\end{proof}

This lemma immediately implies Theorem \ref{t.alpha>2.zero.cap}.  Indeed, for any $m$ the set~$\tilde{S}$ is contained in $\bigcup _{k\geq m}J_k'$, and hence,
\begin{align*}
\Ca(\tilde{S}) \leq \Ca\left( \bigcup _{k\geq m}J_k'\right)\leq \exp\left( -\frac{1}{\sum_{k=m}^\infty 1/|\log r_k'|}\right).
\end{align*}
As $m$ is arbitrary, and the tail sum of a convergent series tends to zero, passing to the limit as $m\rightarrow \infty$ we get the desired
\begin{align*}
\Ca(\tilde{S})=0.
\end{align*}

\begin{proof}[Proof of Theorem~\ref{t:E-G}.]
Assume that $m_{h_0}(E)=R<\infty$. Then for an arbitrarily small $\eps>0$ there exists its cover $\bigcup_j I_j \supset E$ by intervals of length at most $\eps$, such that $\sum_j h_0(|I_j|)=\sum_j \frac{1}{|\log |I_j| \, |} <2R$. Estimate~\eqref{eq:I-bound} then implies, that for any measure $\mu$ on $E$ one has $I(\mu)\ge \frac{1}{2R}$. Moreover, actually~\eqref{eq:I-bound} is a lower bound for the part of the integral $\eps$-close to the diagonal (as $x$ and $y$ can be restricted to belong to the same interval):
\begin{equation}\label{eq:int-eps}
\iint_{|x-y|<\eps} \log |x-y| \, d\mu(x) \, d\mu(y) > \frac{1}{2R}.
\end{equation}
Recall now that $\eps>0$ was arbitrary; if there was a measure $\mu$ on $E$ with $I(\mu)<\infty$, the left hand side of~\eqref{eq:int-eps} 
would tend to zero as $\eps\to 0$. On the other hand, the right hand side is a constant. This contradiction shows that for any measure $\mu$ on $E$ one has $I(\mu)=+\infty$, and thus that $\Ca(E)=0$.
\end{proof}
\begin{remark}
Actually, the statements Lemma~\ref{l:I-ints} and Theorem~\ref{t:E-G} hold in any dimension, with balls replacing the intervals and their diameters taken instead of lengths, and the proofs are the same word for word.
\end{remark}

\begin{proof}[Proof of the first part of Theorem \ref{t.Phase.transition}]
Take the sequence $J_k'$ to be an enumeration of the family $J_{k,n}$. Then,
\begin{align*}
\bigcap_{m=1}^\infty\bigcup_{k=m}^\infty J_k' = \bigcap_{m=1}^\infty\bigcup_{n=m}^\infty V_n;
\end{align*}
for each $n$ there are $n$ intervals $J_k'$ of length $r_n$ (that is, $J_{0,n},\ldots,J_{n-1,n})$, and hence
\begin{equation}\label{e.series.capzero}
\sum_{k=m}^\infty \frac{1}{|\log r_k '|}=\sum_{n=m}^\infty \frac{n}{|\log r_n |}=\sum_{n=m}^\infty \frac{1}{n^{\alpha-1}}
\end{equation}
as $r_n=e^{-n^{\alpha}}$. As for $\alpha>2$ the series \eqref{e.series.capzero} converges, 
$\Ca (S)=0$ due to Theorem \ref{t.alpha>2.zero.cap}.
\end{proof}


\section{Non-continuity of capacity on bounded interval}\label{s.counterexample}

As mentioned previously, in the proof of Theorem \ref{t.subexponential}, there is a tempting shortcut that cannot be taken. It is already known that capacity does not satisfy limit properties that a measure does. In particular, it is not continuous under descending collection of sets. For example, one can take the collection of open bounded sets

$$O_n:=\{z\in \C: 1-\frac{1}{n}<\Im(z)<1 \text{ and } 0<\Re(z)<1\}.$$

\begin{figure}[!h!]
\includegraphics[scale=0.85]{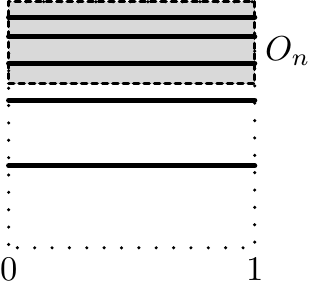} \qquad 
\includegraphics[scale=0.85]{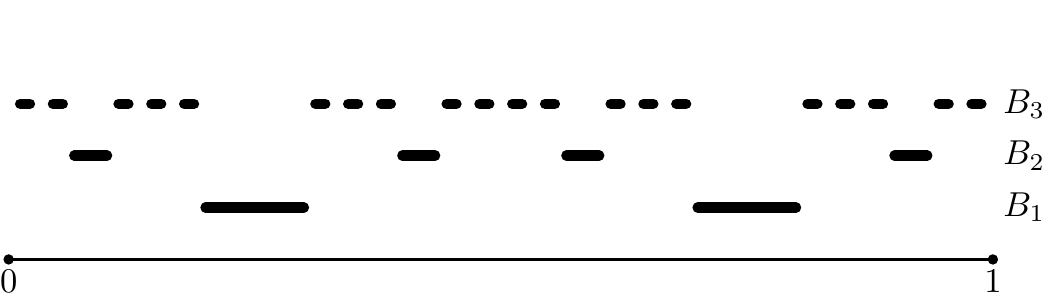}
\caption{Discontinuity of capacity on $[0,1]^2$ and on $[0,1]$}\label{fig:disjoint}
\end{figure}

Then, each $O_n$ contains a translation of the interval $(0,1)$; see Fig.~\ref{fig:disjoint}, left. Hence, $\Ca(O_n)\geq 1/4$. If capacity was continuous on descending open sets, we would have that
\begin{align*}
1/4 \leq \lim_{n\rightarrow \infty} \Ca(O_n)=\Ca\left(\bigcap _{n\in\N} O_n\right)=\Ca(\emptyset)=0.
\end{align*}

A question that appears naturally is whether capacity was continuous under a descending collection of open sets contained in~$[0,1]$? If so, Corollary~\ref{c.union.full.cap} would imply that $\Ca(S)= \Ca(J)$. Unfortunately, the answer to the continuity question on $[0,1]$ is negative, as one can see from the following example.
\begin{example}\label{p.not.continuous}
There exist pairwise disjoint open sets $B_1, B_2, \ldots$ contained in $[0,1]$ with capacity bounded away from $0$. In other words, there exists $\eps>0$ such that 
\begin{align*}
\Ca(B_n)\geq \eps,
\end{align*}
for any $n\in\N$. 
\end{example}

\begin{example}\label{e.main}
There exists a descending sequence $W_1\supset W_2\supset \ldots$ of open sets contained in $[0,1]$ such that
\begin{align*}
\Ca\left(\bigcap_{n\in\N}W_n\right)=0<\eps\leq \Ca\left(W_k\right),
\end{align*}
for some $\eps$ and every $k\in\N$.
\end{example}

Construction of Example \ref{e.main} out of Example \ref{p.not.continuous} is immediate: take 
\begin{align*}
W_m:=\bigcup_{n\geq m}B_n,
\end{align*}
where $B_n's$ are given by Example~\ref{p.not.continuous}. Indeed, one then has $\bigcap_n W_n=\emptyset$, $W_1\supset W_2 \ldots$ by construction, as well as $\Ca (W_n) \geq \Ca(B_n)\geq \eps$. This example shows the discontinuity of example on descending sequence of open subsets of $[0,1]:$ one has $\Ca (\bigcap_n W_n)=0$ while $\lim_{n\rightarrow \infty}  \Ca(W_n)\geq \eps$. Let us pass to the construction proving Example \ref{p.not.continuous}. 

To construct the desired sets $B_n$, consider the unions $V_n=\cup_{i}J_{i,n}$ given by \eqref{eq:V-def}, taking the decreasing speed for the lengths $r_n:=2^{-n}$. Take a subsequence $n_k$ of indices to be defined by $n_1=2^{10},n_k=2^{n_{k-1}+1}$, and define (see Fig.~\ref{fig:disjoint}, right)
\begin{align*}
B_k:=V_{n_k}\setminus \bigcup_{i=1}^{k-1}\overline{V}_{n_i}.
\end{align*}

The sets $B_k$ are then open and disjoint by construction. To show that they satisfy the conclusion of the proposition, it suffices to find probability measures $\nu_k$, supported on $B_k$, such that the energies $I(\nu_k)$ are uniformaly bounded. That is, there exists $C$ such that for all $k$ one has $I(\nu_k)\leq C$. This implies $\Ca (B_k)\geq e^{-C}$, and thus the conclusion of the proposition holds with $\eps=e^{-C}$. 

To do so, first consider the uniform measures $\nu_k^\circ$ on $V_{n_k}$, letting $\nu_k^\circ:= R(\leb|V_{n_k})$, where $\leb$ is the Lebesgue measure on~$[0,1]$. Due to Proposition~\ref{p.redistribution.one.level},

\begin{align*}
I(\nu_{k}^\circ)=I(\leb) + \frac{n_k \log 2}{n_k} +o(1)= 3/2 +\log 2 + o(1).
\end{align*}

Now, let 
\begin{align*}
\nu_k:= \frac{\left.\nu_k^\circ\right|_{B_k}}{ \nu_k^\circ (B_k)  }.
\end{align*}
Then,

\begin{align*}
I(\nu_k):= \frac{1}{ \nu_k^\circ (B_k)^2 }I(\nu_k^\circ),
\end{align*}
so it suffices to check that $\nu_k^\circ(B_k)$ stays bounded away from zero. In fact, we will show that $\nu_k^\circ(B_k)\geq 1/2$. This will follow from a purely geometrical observation: 

\begin{lemma}
\begin{align*}
\leb(V_{n_k}\cap X_k )=\leb(X_k)\cdot\leb(V_{n_k}),
\end{align*}
where 
$$X_k:=[0,1]\setminus \bigcup_{i=1}^{k-1} V_{n_i}.
$$
\end{lemma}
\begin{proof} Note that all the endpoints of $V_{n_i}$, $i=1,\ldots,k-1$ are of the form 
\begin{align*}
\frac{2j+1}{2n_i}\pm \frac{r_{n_i}}{2}=\frac{2j+1}{2^{n_{i-1}+2}}\pm \frac{1}{2^{n_i+1}},
\end{align*}
and hence can be represented as 
\begin{align*}
\frac{a}{2^{n_{k-1}+1}}=\frac{a}{n_k}.
\end{align*}
Hence, $X_k$ is (up to a finite number of points) a union of intervals of the form

\begin{align}\label{e.blocks}
\left(\frac{a}{2^{n_{k-1}+1}},\frac{a+1}{2^{n_{k-1}+1}} \right)=\left(\frac{a}{{n_k}},\frac{a+1}{n_k} \right).
\end{align}
We have 
$$
X_k=\bigcup_{a\in\mathcal{A} } \left(\frac{a}{{n_k}},\frac{a+1}{n_k} \right)\cup P,
$$
where $P$ consists of a finite number of points.
\begin{figure}[!h!]
\includegraphics[width=10cm]{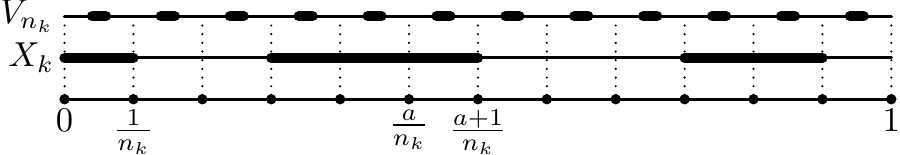}
\caption{The set $V_{n_k}$ and the decomposition into dyadic intervals}
\end{figure}

 On each interval of the form \eqref{e.blocks}, the set $V_{n_k}$ cuts the same measure:
 $$\leb (V_{n_k}\cap [\frac{a}{n_k},\frac{a+1}{n_k}])=r_{n_k}
 $$ 
 and thus the same proportion ${r_{n_k}}\cdot{n_k}$. Hence, 
\begin{align*}
\leb (V_{n_k}\cap X_k)=r_{n_k} \cdot \#\mathcal{A}=\left({r_{n_k}}\cdot{n_k}\right)\left(\frac{\#\mathcal{A}}{ n_k}\right)=\leb(V_{n_k})\cdot\leb(X_k).
\end{align*}
\end{proof}

Due to this lemma, $\nu_k^\circ(B_k)=\leb (X_k).$ On the other hand, 
$$\leb(X_k)\geq 1-\sum_{i=1}^{k-1}\leb (V_{n_i})\geq 1/2.$$
We have obtained the desired $\nu_k^\circ(B_k)\geq 1/2,$ and hence 
$$I(\nu_k)\leq 4(3/2 +\log 2 +o(1)),$$ 
thus concluding the construction.


\end{document}